\documentclass[11pt]{amsart}

\usepackage{amssymb}
\usepackage{amsmath}
\usepackage{amsthm}
\usepackage[all]{xy}

\theoremstyle{plain}
\newtheorem{thm}{Theorem}[section]
\newtheorem{prop}[thm]{Proposition}
\newtheorem{lem}[thm]{Lemma}

\theoremstyle{definition}
\newtheorem{defn}[thm]{Definition}

\theoremstyle{remark}
\newtheorem{rem}[thm]{Remark}

\DeclareMathOperator{\mult}{mult}

\DeclareMathOperator{\Proj}{Proj}

\DeclareMathOperator{\vol}{vol}

\DeclareMathOperator{\lex}{lex}

\def\N{\mathbb{N}}
\def\Z{\mathbb{Z}}
\def\Q{\mathbb{Q}}
\def\R{\mathbb{R}}
\def\C{\mathbb{C}}

\def\r+{\mathbb{R}_{\geq 0}}

\def\ep{\varepsilon}

\def\r+{{\R}_{\geq 0}}
\def\P{\mathbb{P}}
\def\rarw{\rightarrow}

\def\w.{W_{\bullet}}
\def\wdel{W_{\Delta,\bullet}}
\def\*c{\C^{\times}}

\def\rp{{\R}_{+}}
\def\zp{{\Z}_{+}}
\def\mb{\mathbf{m}}

\newcommand{\calo}{\mathcal {O}}

\newcommand{\m}{\mathfrak{m}}

\begin{document}

\title{Okounkov bodies and Seshadri constants}
\author{Atsushi Ito}
\address{Graduate School of Mathematical Sciences, 
the University of Tokyo, 3-8-1 Komaba, Meguro-ku, Tokyo 153-8914, Japan.}
\email{itoatsu@ms.u-tokyo.ac.jp}

\begin{abstract}
Okounkov bodies,
which are closed convex sets defined for big line bundles,
have rich information on the line bundles.
On the other hand,
Seshadri constants are invariants which measure the positivity
of line bundles. 
In this paper,
we prove that Okounkov bodies give lower bounds of Seshadri constants.  
\end{abstract}


\maketitle

%%%%%%%%%%%%%%%%%%%%%%%%%%%%%%%%%%%%%%%%%%%%%%%%

\section{Introduction}

In this paper,
we investigate a relation between
Okounkov bodies and Seshadri constants.

\vspace{3mm}

Okounkov bodies were introduced by Lazarsfeld and Musta\c{t}\u{a} \cite{LM}
and independently by Kaveh and Khovanskii \cite{KK},
based on the work of Okounkov \cite{Ok1,Ok2}.
First, recall the definition of Okounkov bodies.

Let $X$ be a variety of dimension $n$,
$z=(z_1,\ldots,z_n)$ a local coordinate system at a smooth point $p \in X$,
and $>$ a monomial order on $\N^n$.
Then we obtain a valuation
\[
\nu=\nu_{z,>}: \calo_{X,p} \setminus \{0\} \rightarrow \N^n
\]
as follows;
for $f \in  \calo_{X,p} \setminus \{0\}$,
we expand it as a formal power series
\[
f=\sum_{u \in \N^n} c_u z^u
\]
and set
\[
\nu(f) := \min \{ \, u \, | \, c_u \not = 0\},
\]
where the minimum is taken with respect to the monomial order $>$.

Let $L$ be a big line bundle on a projective variety $X$.
Then we can define the Okounkov body $\Delta(L)$ of $L$ as follows.

Fix an isomorphism $L_p \cong \calo_{X,p}$.
Then this isomorphism naturally induces $L^{ k}_p \cong \calo_{X,p}$ for any $k \geq 0$,
and we have the map
\[
H^0(X,kL) \setminus \{0\} \hookrightarrow L^{ k}_p \setminus \{0\} 
\cong \calo_{X,p} \setminus \{0\} \stackrel{\nu}{\to} \N^n.
\]
We write $\nu(H^0(X,kL)) \subset \N^n$ for the image of $H^0(X,kL) \setminus \{0\}$ by this map.
Note that $\nu(H^0(X,kL))$ does not depend on the choice of the isomorphism $L_p \cong \calo_{X,p}$
because $\nu$ maps any unit element
in $\calo_{X,p} $ to $0 \in \N^n$.
Then
\[
\Gamma_{L}=\Gamma_{L,z,>} := \bigcup_{k \in \N} \{k\} \times \nu(H^0(X,kL)) \subset \N \times \N^n
\]
is a semigroup by construction.
We define the \textit{Okounkov body} of $L$ with respect to $z$ and $>$ by
\begin{align*}
\Delta(L)=\Delta_{z,>}(L) :  &=  \text{closed convex hull} \, \Big(\bigcup_{k \geq 1} \frac1k  \nu(H^0(X,kL)) \Big) \\
&=  \Sigma (\Gamma_{L}) \cap (\{1\} \times \R^n),
\end{align*}
where $\Sigma(\Gamma_{L})$ is the closed convex cone spanned by $\Gamma_{L}$.

More generally,
we can define the Okounkov body for a graded linear series.
That is,
for a graded linear series $\w.$ associated to a line bundle $L$ on a (not necessarily projective) variety $X$,
we set
\[
\Gamma_{\w.}=\Gamma_{\w.,z,>} := \bigcup_{k \in \N} \{k\} \times \nu(W_k) \subset \N \times \N^n.
\]
We define the Okounkov body of $\w.$ with respect to $z$ and $>$
by $\Delta(\w.)=\Delta_{z,>} (\w.):=\Sigma (\Gamma_{\w.,z,>}) \cap (\{1\} \times \R^n)$.

\begin{rem}
In \cite{LM},
they use ``admissible flags'' instead of local coordinate systems.
Essentially, there is no difference if $>$ is the lexicographic order.
Local coordinate systems are used in \cite{BC} for instance.
\end{rem}

\vspace{2mm}
On the other hand,
Demailly \cite{De} defined an interesting invariant, Seshadri constant, 
which measures the local positivity of an ample line bundle at a point.
Seshadri constants relate to
jet separation of adjoint bundles,
Ross-Thomas' slope stabilities of polarized varieties \cite{RT},
Gromov width (an invariant in symplectic geometry) \cite{MP}, and so on.
Nakamaye \cite{Na} defined Seshadri constants at very general points 
for big line bundles.
For a detailed treatment of Seshadri constants,
we refer the reader to \cite[Chapter 5]{La} or \cite{BDH+}.

\vspace{1.5mm}
We extend the notion of Seshadri constants slightly.
Usually,
Seshadri constants are defined in a numerical way.
However,
we adopt an equivalent definition in terms of jet separation
since we treat graded linear series in this paper.

Let $\w.$ be a ``birational'' graded linear series associated to a line bundle $L$ on a variety $X$
(see Definition \ref{def of birational} for the definition of a birational graded linear series).
Using jet separation (cf.\ \cite[Theorem 6.4]{De},
\cite[Proposition 6.6]{ELMNP}),
we define the \textit{Seshadri constant} $\ep(\w.;1) $ of $\w.$ at a very general point by
\[
\ep(\w.;1) := \sup_{k > 0} \frac{s(W_k;1)}{k} \in \rp \cup \{+ \infty\},
\]
where $s(W_k;1)$ is the supremum of integers $s \geq -1$
such that the natural map
$ W_k  \hookrightarrow H^0(X,L^{ k})  \rarw L^{ k} \otimes \calo_X / \m_p^{s+1}$ is surjective
for a very general point $p \in X$.
When $L$ is a big line bundle on a projective variety $X$ and $\w.=\{ H^0 (X, kL)\}_k$,
it is easy to show $\w.$ is birational.
Hence we can consider the Seshadri constant of $\w.$ at a very general point,
which we denote by $\ep(L;1)$ or $\ep(X,L;1)$.

To relate Seshadri constants with Okounkov bodies,
we introduce an invariant for a convex set.
For an integral polytope $\Delta \subset \R^n$ of dimension $n$,
we can consider an invariant $\ep(\Delta;1):=  \ep(X_{\Delta},L_{\Delta};1) $,
where $(X_{\Delta},L_{\Delta})$ is the polarized toric variety corresponding to $\Delta$.
To generalize this invariant to an $n$-dimensional convex set $\Delta$,
we introduce a ``monomial'' birational graded linear series $W_{\Delta,\bullet}$ associated to $\calo_{(\*c)^n}$ on $(\*c)^n$
(see Definition \ref{def of monomial g.l.s.}).
Thus we can define $\ep(\Delta;1):= \ep((\*c)^n,W_{\Delta,\bullet};1)$.

\vspace{2mm}
In this paper, we show the following theorem,
which states that Okounkov bodies give lower bounds on Seshadri constants.

\begin{thm}[Special case of Theorem \ref{main thm}]\label{main of special case}
Let $X$ be a projective variety of dimension $n$,
and fix a local coordinate system $z=(z_1,\ldots,z_n)$ on $X$ at a smooth point
and a monomial order $>$ on $\N^n$.
For a big line bundle $L$ on $X$,
it holds that
$\ep(L;1) \geq \ep(\Delta_{z,>}(L);1)$.
\end{thm}

\vspace{1mm}
We explain the idea of the proof of Theorem \ref{main of special case}
when $L$ is ample.
The strategy is to compare two graded linear series
$\{ H^0(X,kL) \}_k$ and $W(\Gamma_L)_{\bullet} := \{  \bigoplus_{u \in \nu(H^0(X,kL))} \C x^u  \}_k $.
Note that $ \bigoplus_{u \in \nu(H^0(X,kL))} \C x^u \subset \C[\N^n]$ is nothing but
the degree $k$-th part of the graded ring $\C[\Gamma_L]$.

If $\Gamma_L$ is finitely generated,
Anderson \cite{An} showed that $\bigoplus_{k} H^0(X,kL)$ degenerates to $\C[\Gamma_L]$.
In other words,
$(X,L)$ degenerates to the (not necessarily normal) toric variety $(\Proj \C[\Gamma_L],\calo(1))$.
Since Seshadri constants have a lower semicontinuity,
we have $\ep(X,L;1) \geq \ep(\Proj \C[\Gamma_L],\calo(1);1) $.
It is easy to check that $\ep(\Proj \C[\Gamma_L],\calo(1);1)=\ep(X_{\Delta(L)},L_{\Delta(L)};1) =\ep(\Delta(L);1)$,
and Theorem \ref{main of special case} follows in this case.

Unfortunately,
 $\Gamma_L$ is not finitely generated in general.
Hence we cannot use the degeneration of the section ring $\bigoplus_{k} H^0(X,kL)$.
Instead, we degenerate linear series as in \cite{CM},
that is,
we degenerate $H^0(X,kL)$ to $W(\Gamma_L)_k=\bigoplus_{u \in \nu(H^0(X,kL))} \C x^u$ for each $k$ separately.
Then we can show $s(H^0(X,kL);1) \geq s(W(\Gamma_L)_k;1)$
and compare the Seshadri constants of $L$ and $W(\Gamma_L)_{\bullet}$.

Thus even if we show Theorem \ref{main of special case}  only for an ample line bundle $L$,
we have to treat the graded linear series $W(\Gamma_L)_{\bullet}$. 
This is the reason why we extend the notion of Seshadri constants to graded linear series.

\vspace{3mm}
The definition of Seshadri constants can be easily generalized to the multi-point case
(cf.\  \cite{La},\cite{BDH+}).
That is,
we can define the Seshadri constants $\ep(\w.;\mb)$ and $\ep(\Delta;\mb)$
for a weight $\mb=(m_1,\ldots,m_r) \in \rp^r$.

Theorem \ref{main of special case} is generalized
to the multi-point case and birational graded linear series as follows.

\begin{thm}[=Theorem \ref{main thm}]\label{main}
Let $\w.$ be a birational graded linear series associated to a line bundle $L$
on an $n$-dimensional variety $X$.
Fix a local coordinate system $z=(z_1,\ldots,z_n)$ on $X$ at a smooth point
and a monomial order $>$ on $\N^n$.
Then $\ep(\w.;\mb) \geq \ep(\Delta_{z,>}(\w.);\mb)$
holds for any $r \in \zp$ and $\mb \in \rp^r$.
\end{thm}

\vspace{2mm}
This paper is organized as follows. In Section 2, we recall some notations and conventions.
In Section 3, we define Seshadri constants of graded linear series and investigate basic properties.
In Section 4, we study Seshadri constants of monomial graded linear series.
In Section 5, we give the proof of Theorem \ref{main}.
Throughout this paper,
we consider varieties or schemes over the complex number field $\C$.

\subsection*{Acknowledgments}
The author wishes to express 
his gratitude to his supervisor Professor Yujiro Kawamata
for his valuable advice, comments and warm encouragement.
He is grateful to Professors Robert Lazarsfeld and Yasunari Nagai
for many valuable comments.
He wishes to thank to Professor Marcin Dumnicki
for answering his questions about references.
He would also like to thank Yoshinori Gongyo, Makoto Miura, Yusuke Nakamura, Shinnosuke Okawa, Taro Sano, and Yusaku Tiba
for helpful discussions and comments.
The author was supported by the Grant-in-Aid for Scientific Research
(KAKENHI No. 23-56182) and the Grant-in-Aid for JSPS fellows.

%%%%%%%%%%%%%%%%%%%%%%%%%%%%%%%%%%%%%%%%%%%%%%%%

\section{Notations and conventions}

We denote by $\N,\Z,\Q,\R$, and $\C$ the set of all
natural numbers, integers, rational numbers, real numbers, and complex numbers respectively.
In this paper, $\N$ contains $0$.
Set $\zp = \N \setminus 0$, $\rp =  \{ x \in \R \, | \, x > 0 \}$.
For a real number $t \in \R$,
the round up of $t$ is denoted by $\lceil t \rceil \in \Z$.

For a subset $S \subset \R^n$, we denote by $\Sigma (S)$ the closed convex cone spanned by $S$.
For $t \in \R$,
we set $tS=\{ \, tu \, | \, u \in S \}$.
For another subset $S' \subset \R^n$,
$S+S'=\{u+u'  \, | \, u \in S, u' \in S'\}$ is the Minkowski sum of $S$ and $S'$.
For simplicity of notation,
we denote $S+(-S')$ by $S-S'$.
For $u' \in \R^n$,
$S+u'=\{u + u' \, | \, u \in S \}$ is
the parallel translation of $S$ by $u'$.

For a convex set $\Delta \subset \R^n$, 
the \textit{dimension} of $\Delta$ is the dimension of the affine space spanned by $\Delta$.
We denote by $\Delta^{\circ}$ the interior of $\Delta$. 

A subset $P \subset \R^n$ is called a \textit{polytope} if it is the convex hull of a finite set in $\R^n$.
A polytope $P$ is \textit{integral}
if all its vertices are contained in $\Z^n$.

For a variety $X$, we say a property holds at a \textit{general} point of $X$
if it holds for all points in the complement of a proper algebraic subset.
A property holds at a \textit{very general} point of $X$
if it holds for all points in the complement of the union of countably many proper subvarieties.

Throughout this paper, a divisor means a Cartier divisor.
Thus we use the words ``divisor'', ``line bundle'', and ``invertible sheaf\,'' interchangeably. 
For divisors $D$ and $D'$,
the inequality $D \geq D'$ means $D-D'$ is effective.

%%%%%%%%%%%%%%%%%%%%%%%%%%%%%%%%%%%%%%%%%%%%%%%%
%def of sc
%%%%%%%%%%%%%%%%%%%%%%%%%%%%%%%%%%%%%%%%%%%%%%%%

\section{Seshadri constants of graded linear series}

In this section,
we define Seshadri constants of graded linear series.

\begin{defn}\label{def of jets separation}
Let $L$ be a line bundle on a (not  necessarily projective) variety $X$,
and $W$  a subspace of $H^0(X,L)$.
For $r \in \zp$ and $\mb =(m_1,\ldots,m_r) \in \Z^r$,
we say that $W$ \textit{separates $\mb$-jets} at smooth $r$ points $p_1,\ldots,p_r$ in $X$
if the natural map
\[
\ W \rightarrow  L  \big/  \Big(L \otimes \bigotimes_{i=1}^r \m_{p_i}^{m_i+1}\Big)
=\bigoplus_{i=1}^r L / \m_{p_i}^{m_i+1} L
\]
is surjective,
where we regard $\m_{p_i}^{m_i+1}=\calo _X$ for $m_i \leq -1$.
We say $W$ \textit{generically separates $\mb$-jets}
if $W$ separates $\mb$-jets at general $r$ points in $X$.
By definition,
any $W$ generically separates $\mb$-jets
if $m_i \leq -1$ for any $i$.

For $W \subset H^0(X,L)$ and $\mb=(m_1,\ldots,m_r) \in \rp^r$,
we define $s(W;\mb) \in \R \cup \{+ \infty\}$ to be
\[
s(W;\mb) = \sup \{ \, t \in \R \, | \, W \ \text{generically separates $ \lceil t \mb \rceil $-jets}\} ,
\]
where $ \lceil t\mb \rceil = (\lceil tm_1 \rceil,\ldots,\lceil tm_r \rceil)$.
When $\dim W < + \infty$, we have
\[
s(W;\mb) = \max \{ \, t \in \R \, | \, W \ \text{generically separates $ \lceil t \mb \rceil $-jets} \} \in \R.
\]
When $W=H^0(X,L)$,
we denote $s(W,\mb)$ by $s(L;\mb)$.
\end{defn}

Suppose $X$ is a variety of dimension $n$
and $L$ is a line bundle on $X$.
Let $\w.=\{W_k\}_{k \in \N}$ be a \textit{graded linear series} associated to $L$,
i.e., $W_k$ is a subspace of $H^0(X,kL)$ for any $k \geq 0$ with $W_0=\C$,
such that
$\bigoplus_{k \geq 0} W_k$ is a graded subalgebra of the section ring $\bigoplus_{k \geq 0} H^0(X,kL)$. 
When all $W_k$ are finite dimensional,
$\w.$ is called \textit{of finite dimensional type}. 

\begin{defn}\label{def of birational}
A graded linear series $\w.$ on a variety $X$ is \textit{birational} if
the function field $K(X)$ of $X$ is generated by $\{ \, f / g \in K(X) \, | \, f,g \in W_k, g \not = 0  \}$ over $\C$
for any $k \gg 0$.
When $\w.$ is of finite dimensional type,
this is clearly equivalent to the condition that
the rational map defined by $|W_k|$ is birational onto its image for any $k \gg 0$,
which is Condition (B) in \cite[Definition 2.5]{LM}.  
\end{defn}

Now we define Seshadri constants of graded linear series.

\begin{defn}\label{def of sc}
Let $\w.$ be a birational graded linear series associated to a line bundle $L$ on a variety $X$.
For $\mb=(m_1,\ldots,m_r) \in \rp^r$,
we define the \textit{Seshadri constant} of $\w.$ at very general points with weight $\mb$ to be
\[
\ep(X,\w.;\mb)=\ep(\w.;\mb):=\sup_{k > 0} \dfrac{s(W_k;\mb)}k \in \rp \cup \{+\infty\}.
\]
Note that $s(W_k;\mb) > 0$ holds for $k \gg0$ by the birationality of $\w.$.

When $W_k=H^0(X,kL)$ for any $k$,
we denote it by $\ep(X,L;\mb)$ or $\ep(L;\mb)$.
\end{defn}

\begin{rem}
The definition of Seshadri constants by jet separation is due to Theorem 6.4 in \cite{De}.
See also \cite{La} or \cite{ELMNP}.
We treat the definition by blowing ups in Lemmas \ref{usual def} and \ref{def by resolution}.
\end{rem}

\begin{rem}\label{rem_change of bundles}
By definition,
$s(W;\mb) \leq s(W';\mb)$ holds for
subspaces $W \subset W'$ in $H^0(X,L)$.
Thus $\ep(\w.;\mb) \leq \ep(\w.^{'};\mb)$ holds for $\w.,\w.^{'}$ associated to $L$
if $W_k \subset W_k^{'}$ for any $k$
(we write $\w. \subset \w.^{'}$ for such $\w.,\w.^{'}$).

For an injection $L \hookrightarrow L'$ between line bundles on $X$
and a subspace $W \subset H^0(X,L)$,
$s(W;\mb)$ does not change if we regard $W$ as a subspace of $H^0(X,L')$
because we consider jet separation at general points.
Hence $\ep(\w.;\mb )$ does not change for $\w.$ associated to $L$
if we consider that $\w. $ is associated to $L'$. 

Let $\pi:X' \rarw X$ be a birational morphism
and $\w.$ a graded linear series associated to $L$ on $X$.
Then we can consider that $\w.$ is a graded linear series on $X'$ associated to $\pi^*L$
by the natural inclusion $H^0(X,kL) \subset H^0(X',k \pi^*L)$.
By a similar reason to above,
we have $\ep(X,\w.;\mb)=\ep(X',\w.;\mb)$.
\end{rem}

By the following lemma,
we may assume that $\w.$ is of finite dimensional type in many cases
when we prove properties of $\ep(X,\w.;\mb)$.

\begin{lem}\label{increasing seq of g.l.s.}
Let $\w.$ be a birational graded linear series.
Suppose that $ W_{1,\bullet} \subset W_{2,\bullet} \subset \cdots \subset  W_{l,\bullet} \subset \cdots \subset \w.$
is an increasing sequence of birational graded linear series in $\w.$
such that $ W_k = \bigcup_{l =1}^{\infty} W_{l,k}$ for each $k$.
Then it holds that $\ep(\w.;\mb) = \sup_l  \ep(W_{l, \bullet};\mb)
= \lim_l  \ep(W_{l, \bullet};\mb)$.
\end{lem}

\begin{proof}
Since $\ep(W_{l, \bullet};\mb)$ is monotonically increasing,
$\lim_l  \ep(W_{l, \bullet};\mb)$ exists.
The inequalities $\ep(\w.;\mb) \geq \sup_l  \ep(W_{l, \bullet};\mb)
\geq \lim_l  \ep(W_{l, \bullet};\mb)$ are clear.
Thus it is enough to show $\ep(\w.;\mb) \leq \lim_l  \ep(W_{l, \bullet};\mb)$.

Fix $k$ and a real number $t < s(W_k;\mb)$.
By definition,
$W_k$ generically separates $\lceil t\mb \rceil$-jets.
Hence $W_{l,k}$ also generically separates $\lceil t\mb \rceil$-jets
for $l \gg 0$ by the assumption $W_k = \bigcup_l W_{l,k}$.
Thus it holds that $s(W_k;\mb)/ k = \lim_l s(W_{l,k};\mb) /k  \leq \lim_l \ep(W_{l,\bullet};\mb)$.
By the definition of $\ep(\w.;\mb)$,
we have $\ep(\w.;\mb) \leq \lim_l  \ep(W_{l, \bullet};\mb)$.
\end{proof}

In Definition \ref{def of sc},
$\ep(\w.;\mb)$ is defined by the supremum,
but in fact it is the limit.

\begin{lem}\label{sup=lim for sc}
In Definition \ref{def of sc},
$\ep(\w.;\mb) = \lim \dfrac{s(W_k;\mb)}k$ holds.

\end{lem}

\begin{proof}
By Lemma
\ref{increasing seq of g.l.s.},
we may assume that $\w.$ is of finite dimensional type.
To prove this lemma,
it suffices to show that a sequence $\{ s(W_k;\mb)\}_k$ is superadditive,
i.e.,
\[
s(W_{k+l};\mb) \geq s(W_k;\mb) + s(W_l;\mb)
\]
holds for $k,l >0$ if $  s(W_k;\mb), s(W_l;\mb) \geq 0$.

For simplicity, we set $s_k=s(W_k;\mb)$ in this proof.
We prove the superadditivity only when $r=1$,
and write $m \in \rp$ instead of $\mb$.
When $r > 1$,
the proof is essentially the same.
Hence we leave the details to the reader.

Fix a very general point $p \in X$.
For each $k,i \geq 0$, set
\[
W_{k,i}=W_{k} \cap H^0(X,L^{ k} \otimes \m_{p}^i) \subset H^0(X,L^{ k}).
\]
For any $s \in \N$,
it is easy to show that
\[
W_{k} \rightarrow L^{ k}\otimes \calo _X/\m_{p}^{s+1}
\]
is surjective
if and only if so is
\[
W_{k,i} \rightarrow L^{ k} \otimes \m_{p}^i/\m_{p}^{i+1}
\]
for each $i \in \{0,1,\ldots,s \}$.

Fix an integer $0 \leq i \leq  \lceil  s_k m \rceil + \lceil  s_l m \rceil$.
Then there exist integers $0 \leq i_1  \leq \lceil s_k m \rceil , 0 \leq i_2  \leq \lceil s_l m \rceil $ such that $i=i_1+ i_2$.
Consider the following diagram
\[\xymatrix{
W_{k,i_1} \otimes W_{l,i_2}
\ar[d] \ar[r]^(0.32){\alpha} \ar@{}[dr]|{\circlearrowleft} 
&L^{ k} \otimes \m_{p}^{i_1}/\m_{p}^{i_1+1} \otimes L^{ l} \otimes \m_{p}^{i_2}/\m_{p}^{i_2+1}  \ar[d]^{\beta} \\
W_{k+l,i} \ar[r]^(0.4){\gamma} & L^{ k+l} \otimes \m_{p}^i/\m_{p}^{i+1}. \\
}\]
In this diagram,
$\alpha$ is surjective because $i_1  \leq \lceil s_k m \rceil , i_2  \leq \lceil s_l m \rceil $,
and $\beta$ is clearly surjective.
Hence $\gamma$
is also surjective for any $i \in \{0,1,\ldots,\lceil  s_k m \rceil + \lceil  s_l m \rceil\}$.
Thus $W_{k+l}$ generically separates $\lceil  s_k m \rceil + \lceil  s_l m \rceil$-jets,
which means
\[
s_{k+l} \geq (\lceil  s_k m \rceil + \lceil  s_l m \rceil)/m \geq  s_k + s_l . \qedhere
\]
\end{proof}

Many properties of Seshadri constants of ample line bundles 
also hold for graded linear series.

\begin{lem}\label{easy property of sc}
For a birational graded linear series $\w.$ (resp.\ $W_{\bullet}^{'}$)
associated to a line bundles $L$ (resp.\  $L'$) on a variety $X$
and $\mb \in \rp^r$,
it holds that
\begin{itemize}
\item[(1)] $\ep(\w.^{(l)};\mb)=l \cdot \ep(\w.;\mb)$ for $l \in \zp$,
where $\w.^{(l)}$ is the graded linear series associated to $L^{ l}$
defined by $W_k^{(l)} : =W_{kl}$. 
\item[(2)] $\ep(\w.;t\mb)=t^{-1}\ep(\w.;\mb)$ for $t  \in \rp$.
\item[(3)] $\ep(W_{\bullet}^{''};\mb) \geq \ep(\w.;\mb)+\ep(W_{\bullet}^{'};\mb)$,
where $W_{\bullet}^{''}$ is the graded linear series associated to $L \otimes L'$
defined by $W_k^{''}:= \text{the image of } \\
W_k \otimes W_k^{'} \rightarrow H^0(X,k(L \otimes L'))$.
\item[(4)] $\ep(\w.;\mb) \leq
\sqrt[n]{{\vol(\w.)} / |\mb|_n}$,
where $n=\dim X$, $\vol(\w.)=\lim_k \dfrac{\dim W_k}{k^n / n!}  $,
and $|\mb|_n=\sum_{i=1}^r m_i^n$.
\end{itemize}
\end{lem}

\begin{proof}
By Lemma \ref{sup=lim for sc}, we have
\begin{align*}
\ep(\w.^{(l)};\mb)&= \lim_k \dfrac{s(W_k^{(l)};\mb)}k\\
&= \lim_k \dfrac{s(W_{kl};\mb)}k\\
&= l \cdot \lim_k \dfrac{s(W_{kl};\mb)}{kl}\\
&= l \cdot \lim_k \dfrac{s(W_k;\mb)}k=l \cdot \ep(\w.;\mb).
\end{align*}
Hence (1) is shown.

We can show (2) easily from the definition and the following.
\begin{align*}
s(W_k;t\mb) &= \sup \{ \, s \in \R \, | \, W_k \ \text{generically separates $ \lceil st \mb \rceil $-jets} \, \} \\
          &= t^{-1} \sup \{ \, s \in \R \, | \, W_k \ \text{generically separates $ \lceil s \mb \rceil $-jets} \, \}\\
          &= t^{-1} s(W_k;\mb).
\end{align*}

To prove (3),
it suffices to show $s(W_k^{''};\mb) \geq s(W_k;\mb)+s(W_k^{'};\mb)$ for any $k$.
We can show this
by the argument similar to the proof of Lemma \ref{sup=lim for sc}.
We leave the details to the reader.

To show (4),
we fix a positive number $0 < t < \ep(\w.;\mb)$.
By Lemma \ref{sup=lim for sc},
the inequality $s(W_k;\mb) > kt$ holds for any $k \gg 0$.
Thus $W_k$ separates $\lceil kt \mb \rceil$-jets for any $k \gg 0$.
In other words,
\[
W_k \rarw \bigoplus_i L \otimes \calo/\mathfrak{m}_{p_i}^{\lceil ktm_i \rceil +1}
\]
is surjective for very general $p_1,\ldots,p_r$.
This surjection induces
\begin{align*}
\dim W_k & \geq  \dim \bigoplus_i L \otimes \calo/\mathfrak{m}_{p_i}^{\lceil ktm_i \rceil +1} \\
&= \sum_i {\lceil ktm_i \rceil+n \choose n}=\sum_i \dfrac{t^n m_i^n}{n!} k^n + O(k^{n-1}).
\end{align*}
Thus we have $\vol (\w.) \geq t^n |\mb|_n$ by $k \rarw +\infty$.
We finish the proof by $t \rarw \ep(\w.;\mb)$.
\end{proof}

Definition \ref{def of sc}
is a natural generalization of the well-known definition of the Seshadri constant (at very general points)
for a nef and big line bundle (cf. \cite[Theorem 5.1.17]{La}).

\begin{lem}\label{usual def}
For a nef and big line bundle $L$ on a projective variety $X$
and $\mb=(m_1,\ldots,m_r) \in \rp^r$,
it holds that
\[
\ep(X,L;\mb) = \max \{ \, t \geq 0 \, | \,  \mu^{*}L- t \, \sum_{i=1}^{r}m_i E_i \ \text{\rm{is nef}} \,  \},
\]
where $\mu: \widetilde{X} \rightarrow X$ is the blowing up along very general $r$ points on $X$
and $E_1,\ldots,E_r$ are the exceptional divisors.
\end{lem}

\begin{proof}
The proof is essentially the same as that of \cite[Theorem 5.1.17]{La}.
First,
we show $\ep(X,L;\mb) \leq \max \{ \, t \geq 0 \, | \,  \mu^{*}L- t \, \sum_{i=1}^{r}m_i E_i \ \text{is nef} \,  \} $,
i.e., $\mu^{*}L- \ep(X,L;\mb)  \sum_{i=1}^{r}m_i E_i $ is nef.
Fix a curve $C \subset X$
and let $\widetilde{C} \subset \widetilde{X}$ be the strict transform of $C$.
It is enough to show $ (\mu^{*}L- \ep(X,L;\mb)  \sum_{i=1}^{r}m_i E_i) .\widetilde{C} \geq 0$.
For each $k$,
the line bundle $L^{ k} $ separates $\lceil s_k \mb \rceil$-jets
at very general points $p_1,\ldots,p_r$,
where $s_k:=s(kL;\mb)$.
Hence there exists an effective divisor
$D \in |L^{ k} \otimes \bigotimes_i \mathfrak{m}_{p_i}^{\lceil s_k m_i \rceil} |$
such that $D$ does not contain $C$.
Thus it holds that
\begin{align*}
\mu^{*}L .\widetilde{C}
= L.C 
&= k^{-1} D.C  \\
&\geq k^{-1} \sum_i \mult_{p_i}(D) \cdot \mult_{p_i}(C) \\
&\geq k^{-1} \sum_i \lceil s_k m_i \rceil \mult_{p_i}(C) \\
&\geq k^{-1} \sum_i  s_k m_i \mult_{p_i}(C) 
= k^{-1}  s_k \sum_i m_i E_i.\widetilde{C},
\end{align*}
where $\mult_{p_i}$ is the multiplicity at $p_i$. 
By $k \rarw +\infty $,
we have $\mu^{*}L .\widetilde{C} \geq \ep(X,L;\mb) \sum_{i=1}^{r}m_i E_i.\widetilde{C}$.

\vspace{2mm}
We show the opposite inequality.
First,
assume $L$ is ample.
Let $p_1,\ldots,p_r$ be very general $r$ points in $X$.
Fix a rational number $0 < t = a / b < \max \{ \, t \geq 0 \, | \,  \mu^{*}L- t \, \sum_{i=1}^{r}m_i E_i \ \text{is nef} \,  \} $
with positive integers $a,b$.
Then $b\mu^{*}L- a  \sum_{i=1}^{r}m_i E_i$ is an ample $\R$-line bundle on $\widetilde{X}$.
Multiplying $a$ and $b$ by a sufficiently large positive integer,
we may assume $b\mu^{*}L- \sum_{i=1}^{r} \lceil am_i \rceil E_i$ is ample. 
By Serre's vanishing theorem,
there exists a natural number $N$ such that
\[
H^1(X,kb L \otimes \bigotimes_i \mathfrak{m}_{p_i}^{k \lceil am_i \rceil} )
=H^1(\widetilde{X}, k (b\mu^{*}L- \sum_{i} \lceil am_i \rceil E_i))=0
\]
for every $k \geq N$.
This means $kbL$ separates $(k \lceil am_1 \rceil-1,\ldots,k \lceil am_r \rceil-1)$-jets
at $p_1,\ldots,p_r$,
that is,
\[
\dfrac{ s(kbL;\mb)}{kb} \geq  \min_i \dfrac{k \lceil am_i \rceil-1}{kb m_i} .
\]
By $k \rarw +\infty$,
we have
\[
\ep(L;\mb) \geq \lim_k  \min_i \dfrac{k \lceil am_i \rceil-1}{kb m_i} = \frac{a}{b}=t.
\]
By $t \rarw \max \{ \, t \geq 0 \, | \,  \mu^{*}L- t \, \sum_{i=1}^{r}m_i E_i \ \text{is nef} \,  \} $,
we obtain $\ep(L;\mb) \geq \max \{ \, t \geq 0 \, | \,  \mu^{*}L- t \, \sum_{i=1}^{r}m_i E_i \ \text{is nef} \,  \}$.

Next we show the nef and big case.
Since $L$ is nef and big,
there exists an effective divisor $E$ on $X$
such that $L- sE$ is ample for any $0 < s \ll1$.
Fix a rational number $0 < s \ll1$
and take a sufficiently divisible integer $l \in\zp$ such that $ls \in\Z$.
Then $\ep(l(L-sE);\mb) \leq \ep(lL;\mb)=l \cdot \ep(L;\mb)$ holds
by Remark
\ref{rem_change of bundles} and Lemma
\ref{easy property of sc} (1).
Since $l(L-sE)$ is ample,
we have
\begin{align*}
\ep(l(L-sE);\mb) &= \max \{ \, t \geq 0 \, | \,  \mu^{*}(l(L-sE))- t \, \sum_{i} m_i E_i \ \text{is nef} \,  \} \\
&= l \cdot \max \{ \, t \geq 0 \, | \,  \mu^{*}(L-sE)- t \, \sum_{i} m_i E_i \ \text{is nef} \,  \}.
\end{align*}
Hence
$ \ep(L;\mb) \geq \max \{ \, t \geq 0 \, | \,  \mu^{*}(L-sE)- t \, \sum_{i=1}^{r}m_i E_i \ \text{is nef} \,  \}$
holds.
By $s \rarw 0$,
we have 
$ \ep(L;\mb) \geq \max \{ \, t \geq 0 \, | \,  \mu^{*}L- t \, \sum_{i=1}^{r}m_i E_i \ \text{is nef} \,  \}$.
\end{proof}

For a line bundle $L$ on a projective variety $X$,
$\w.=\{H^0(X,kL)\}_k$ is birational
if and only if $L$ is big.
For a nef but not big line bundle $L$ on $X$,
we define $\ep(X,L;\mb):=\max \{ \, t \geq 0 \, | \,  \mu^{*}L- t \, \sum_{i=1}^{r}m_i E_i \ \text{is nef} \,  \}=0$.

For projective varieties,
we can describe Seshadri constants of graded linear series
by using those of nef line bundles as follows.
This is a simple generalization of (some part of) \cite[Propositions 6.4, 6.6]{ELMNP},
which treat $\ep(L;1)$ for a big line bundle $L$,
although their notations are slightly different.

\begin{lem}\label{def by resolution}
Let $\w.$ be a birational graded linear series associated to a line bundle $L$ on a projective variety $X$.
For each $k > 0$,
set
\[
M_k=\mu_k^{*} (kL) - F_k,
\]
where $\mu_k : X_k \rightarrow X$ is a resolution of the base ideal
\[
\mathfrak{b}_k:= \text{\rm{the image of}} \, \ W_k \otimes L^{-k} \rightarrow \calo_X 
\]
and $\calo_{X_k}(-F_k):=\mu_k^{-1} \mathfrak{b}_k$.
Set $M_k=0$ if $W_k=0$.
Then it holds that
\[
\ep(X,\w.;\mb)=\sup_{k > 0} \dfrac{\ep(X_k,M_k;\mb)}k=\lim_{k > 0} \dfrac{\ep(X_k,M_k;\mb)}k.
\]
\end{lem}

\begin{proof}
By definition,
$M_k$ is nef and $\ep(M_k,\mb) \geq 0$ for $k > 0$.
To prove the existence of $\lim_k \ep(X_k,M_k;\mb) / k$
and the second equality in the above statement,
it suffices to show the superadditivity of $\{ \ep(X_k,M_k;\mb) \}_k$,
i.e.,
\begin{align}
\tag{$\dagger$}
\ep(X_{k+l}, M_{k+l},\mb) \geq  \ep(X_k,M_k;\mb) + \ep(X_l,M_l;\mb)
\end{align}
for $k, l > 0$ such that $W_k,W_l \not = 0$.

To show ($\dagger$),
fix such $k , l > 0$.
We can take a common resolution of $\mathfrak{b}_k, \mathfrak{b}_l$, and $\mathfrak{b}_{k+l}$.
Since $M_{k+l} \geq M_k + M_l$, ($\dagger$) follows from Lemmas \ref{easy property of sc} (3) or \ref{usual def}.

\vspace{1.3mm}
Next we show the first equality.
Since $\mu_k^{*}|W_k| \subset |M_k|$,
it holds that $s(W_k;\mb) \leq  s(M_k;\mb) \leq \ep(M_k;\mb)$ for any $k$.
Thus we have
\[
\ep (X,\w.;\mb) =\lim  \dfrac{s(W_k;\mb)}k \leq \lim \dfrac{\ep(M_k;\mb)}k .
\]

To show the opposite inequality,
we use Lemma \ref{usual def}.
Since $\w.$ is birational,
the morphism $\varphi_k : X_k \rightarrow \P^{\dim |W_k|}$ defined by $\mu_k^{*}|W_k|$
is birational onto its image for $k \gg 0$.
Denote the image of $\varphi_k$ by $Y_k$.
By Lemma \ref{usual def},
$\ep(X_k,M_k;\mb)=\ep(Y_k,\calo_{Y_k}(1);\mb)$ holds
because $\varphi_k$ is birational.
Furthermore $W_k^l=H^0(Y_k,\calo_{Y_k}(l))$ for $l \gg 0$,
where $W_k^l$ is the image of $W_k^{\otimes l} \rightarrow W_{kl}$.
This implies $s(\calo_{Y_k}(l);\mb) = s(W_k ^l;\mb) \leq s(W_{kl};\mb)$ for $l \gg 0$.
Thus we have
\begin{align*}
\dfrac{\ep(X_k,M_k;\mb)}k
&= \dfrac{\ep(Y_k,\calo_{Y_k}(1);\mb)}k \\
&= \frac1k \lim_l \dfrac{s(\calo_{Y_k}(l);\mb)}l \\
&\leq  \lim \dfrac{s(W_{kl};\mb)}{kl}= \ep(\w.;\mb). \qedhere
\end{align*}
\end{proof}

%%%%%%%%%%%%%%%%%%%%%%%%%%%%%%%%%%%%%%%%%%%%%%%%%%%%%%%%%%%%%%%%%%%%%%%

\section{Monomial graded linear series on $(\*c)^n$}

\begin{defn}\label{def of monomial g.l.s.}
Let $n$ be a positive integer.
For a subset $S \subset \R^n$,
we set
\[
V_S:=\bigoplus_{u \in S \cap \Z^n}\C x^u,
\]
which is a subspace of $\bigoplus_{u \in \Z^n}\C x^u=H^0((\C^{\times })^n,\calo_{(\C^{\times })^n}).$

For a convex set $\Delta$ in $\R^n$,
we define the \textit{monomial graded linear series} $W_{\Delta,\bullet}$ associated to $\calo_{(\*c)^n}$ by
\[
W_{\Delta,k}:=V_{k\Delta} \subset H^0((\*c)^n,\calo_{(\*c)^n}).
\]
It is easy to see that $\wdel$ is birational
if and only if $\dim \Delta=n$.
\end{defn}

\begin{defn}\label{def_of_ep_delta}
For an $n$-dimensional convex set  $\Delta \subset \R^n$ and $\mb \in \rp^r$,
we define $\ep(\Delta;\mb):=\ep(W_{\Delta,\bullet},\mb) \in \rp$.
\end{defn}

\begin{rem}
For an integral polytope $\Delta \subset \R^n$ of dimension $n$,
clearly $\ep(\Delta;\mb) = \ep(X_{\Delta},L_{\Delta};\mb)$ holds,
where $(X_{\Delta},L_{\Delta})=(\Proj \bigoplus_k V_{k \Delta}, \calo(1))$
is the polarized toric variety corresponding to $\Delta$.
\end{rem}

We show some basic properties of $\ep(\Delta;\mb)$ in this section.

\begin{lem}\label{property1 of s}
The following hold for subsets $S_1,S_2$ and $n$-dimensional convex sets $\Delta_1, \Delta_2 $ in $\R^n$
such that $S_1 \subset S_2, \, \Delta_1 \subset \Delta_2$.
\begin{itemize}
\item[(1)] $ s(V_{S_1};\mb) \leq s(V_{S_2};\mb), \,
\ep(\Delta_1;\mb) \leq \ep(\Delta_2;\mb)$.
\item[(2)]$ s(V_{S_1+u};\mb)= s(V_{S_1};\mb)$ for $u \in \Z^n$.
\item[(3)] $\ep(t\Delta_1;\mb)=t \cdot \ep(\Delta_1;\mb), \,
\ep(\Delta_1 +u;\mb)=\ep(\Delta_1;\mb)$
for $t \in \rp \cap \Q$ and $u \in \Q^n$.
\end{itemize}

\end{lem}

\begin{proof}
(1) is clear.
(2) immediately follows from the diagram
\[\xymatrix{
V_{S_1} \ar[d]_{\wr} \ar@{^{(}->}[r] \ar@{}[dr]|\circlearrowleft & \C[x_1^{\pm 1},\ldots,x_n^{\pm 1}] \ar[d]^{\times x^u}_{\wr } \\
V_{S_1+u} \ar@{^{(}->}[r] & \C[x_1^{\pm 1},\ldots,x_n^{\pm 1}] . \\
}\]
\\ 
To show (3),
choose sufficiently divisible $l \in \zp$ such that $lt \in \Z$ and $lu \in \Z^n$.
Then we have
\begin{align*}
\ep(t\Delta_1;\mb)&= \lim_k \dfrac{s(V_{kt\Delta_1};\mb)}k\\
&= \lim_k \dfrac{s(V_{klt\Delta_1};\mb)}{kl}\\
&= t \lim_k \dfrac{s(V_{klt\Delta_1};\mb)}{klt}=t \cdot \ep(\Delta_1;\mb)
\end{align*}
and
\begin{align*}
\ep(\Delta_1 +u;\mb)&= \lim_k \dfrac{ s(V_{k(\Delta_1+u)};\mb)}k\\
&= \lim_k \dfrac{ s(V_{kl(\Delta_1+u)};\mb)}{kl}\\
&= \lim_k \dfrac{ s(V_{kl\Delta_1+klu};\mb)}{kl}\\
&= \lim_k \dfrac{ s(V_{kl\Delta_1};\mb)}{kl}= \ep(\Delta_1;\mb).
\end{align*}
The last but one equality follows from (2) since $klu \in \Z^n$.
\end{proof}

\begin{lem}\label{interior}
For an $n$-dimensional convex set $\Delta \subset \R^n$,
$\ep(\Delta;\mb)=\ep(\Delta^\circ;\mb)$ holds.
\end{lem}

\begin{proof}
It is enough to show $\ep(\Delta;\mb) \leq \ep(\Delta^\circ;\mb) $.
Fix $u \in \Delta^\circ \cap \Q^n$.
By the convexity of $\Delta$,
we have $\Delta - u \subset t  (\Delta^\circ - u)$ for $t >1$.
Thus it holds that
\[
\ep(\Delta;\mb)=\ep(\Delta-u;\mb)
\leq \ep(t(\Delta^\circ -u);\mb) = t \cdot \ep(\Delta^\circ;\mb)
\]
for $t >1$ in $\Q$ by Lemma \ref{property1 of s}. 
By $t \rightarrow 1$,
the lemma is proved.
\end{proof}

The property (3) in Lemma \ref{property1 of s} holds
for any $t \in \rp$ and $u \in \R^n$.

\begin{lem}\label{for R}
For an $n$-dimensional convex set $\Delta \subset \R^n$,
$u \in \R^n$, and $t \in \rp$,
it holds that
\[
\ep(\Delta +u;\mb)=\ep(\Delta;\mb),
\quad \ep(t\Delta;\mb)=t \cdot \ep(\Delta;\mb).
\]
\end{lem}

\begin{proof}
Fix $u' \in \Delta^\circ \cap \Q^n$.
As in the proof of Lemma \ref{interior},
$\Delta - u' \subset (1+t') (\Delta^\circ - u') $
holds for $t' > 0$.
Translating the convex sets by $u+u'$,
we have $\Delta + u \subset (1+t') (\Delta^\circ - u') + u+u' $.
Choose $u'' \in \Q^n$ such that $ (u+u') - u'' \in t' (\Delta^\circ - u')$.
Then we have $\Delta + u \subset (1+2t') (\Delta^\circ - u') + u'' $.
Hence for $t' \in \rp \cap \Q$,
it holds that
\[
\ep(\Delta+u;\mb) \leq \ep((1+2t') (\Delta^\circ - u') + u'';\mb) = (1+2t') \ep(\Delta;\mb)
\]
by Lemmas \ref{property1 of s} (3) and \ref{interior}.
Thus we obtain $\ep(\Delta +u;\mb) \leq \ep(\Delta;\mb)$ by $t' \rightarrow 0$.
Since the opposite inequality $\ep(\Delta +u;\mb) \geq \ep(\Delta;\mb)$ also holds similarly,
$\ep(\Delta +u;\mb) = \ep(\Delta;\mb)$ follows.
\smallskip

For $t_1,t_2 \in \Q$ such that $0 < t_1 \leq t \leq t_2 $,
we have inclusions $t_1(\Delta-u') \subset t(\Delta-u') \subset t_2(\Delta-u')$.
Thus it holds that
\[
\ep(t_1(\Delta-u');\mb) \leq \ep(t(\Delta-u');\mb) \leq \ep(t_2(\Delta-u');\mb).
\]
By Lemma \ref{property1 of s} (3) and the first equality of this lemma, which we have already proved,
$\ep(t(\Delta-u');\mb)=\ep(t\Delta;\mb)$
and
$\ep(t_i(\Delta-u');\mb) = t_i \cdot \ep(\Delta;\mb)$
for $i=1,2$.
Combining these inequalities,
we have
\[
t_1 \cdot \ep(\Delta;\mb) \leq \ep(t\Delta;\mb) \leq t_2 \cdot \ep(\Delta;\mb).
\]
By $t_1,t_2 \rightarrow t$,
$\ep(t\Delta;\mb)=t \cdot \ep(\Delta;\mb)$ follows.
\end{proof}

\begin{lem}\label{s=lim s_i}
Let $\Delta_1 \subset \Delta_2 \subset \cdots \subset \Delta_i \subset \cdots$
be an increasing sequence of $n$-dimensional convex sets in $\R^n$,
and set $\Delta=\bigcup_{i = 1}^{\infty} \Delta_i$.
Then it holds that $\ep(\Delta;\mb)=\sup_i \ep(\Delta_i;\mb)=\lim_i \ep(\Delta_i;\mb)$
for any $\mb \in \rp^r$. 
\end{lem}

\begin{proof}
This lemma follows from Lemma \ref{increasing seq of g.l.s.} immediately.
\end{proof}

%%%%%%%%%%%%%%%%%%%%%%%%%%%%%%%%%%%%%%%%%%%%%%%%%%%%%%%%%%%%%%%%%%%%%%%%%%%%%%%%%%%%%%%%%%%

%%%%%%%%%%%%%%%%%%%%%%%%%%%%%%%%%%%%%%%%%%%%%%%%

\section{Okounkov bodies and Seshadri constants}\label{Okounkov bodies and Seshadri constants}

In this section,
we prove Theorem \ref{main}.

\subsection{Preliminary}

For a subsemigroup $\Gamma$ in $\N \times \N^n$ and $k \in \N$,
set
\begin{align*}
\Delta(\Gamma) &= \Sigma(\Gamma) \cap (\{1\} \times \R^n),\\
\Gamma_k &= \Gamma \cap (\{k\} \times \N^n).
\end{align*}
Recall that $\Sigma(\Gamma)$ is the closed convex cone in $\R \times \R^n$ spanned by $\Gamma$.
We regard $\Delta(\Gamma)$ and $\Gamma_k$ as subsets in $\R^n$ and $\N^n$ respectively in a natural way.

\begin{defn}
For a semigroup $\Gamma$ in $\N \times \N^n$,
we define a graded linear series $W(\Gamma)_{\bullet}$ on $(\*c)^n$ associated to $\calo_{(\*c)^n}$ by
$W(\Gamma)_{k} := V_{\Gamma_k}.$
\end{defn}

The birationality of $W(\Gamma)_{\bullet}$ is interpreted as the following conditions of $\Gamma$.

\begin{defn}\label{def of bir for semigroup}
A semigroup $\Gamma$ in $\N \times \N^n$ is \textit{birational} if
\begin{itemize}
\item[i)] $\Gamma_0=\{0\} \in \N^n $,
\item[ii)] $\Gamma$ generates $\Z \times \Z^n$ as a group.
\end{itemize}
These conditions are (2.3) and (2,5) in \cite{LM} respectively.
It is easy to check that $\Gamma$ is birational
if and only if so is
the graded linear series $W(\Gamma)_{\bullet}$.
\end{defn}

Let $>$ be a \textit{monomial order} on $\N^n$,
i.e.,
$>$ is a total order on $\N^n$ such that
(i) for every $u \in \N^n \setminus 0$, $u > 0$ holds,
and (ii) if $ v > u$ and $w \in \N^n$, then $w+v > w+u$.
In this paper,
$v > u$ does not contain the case $v=u$.

Let $X$ be a variety of dimension $n$
and  $z=(z_1,\ldots,z_n)$ a local coordinate system at a smooth point $p \in X$.
For a birational graded linear series $\w.$ associated to a line bundle $L$ on $X$,
we can define a semigroup $\Gamma_{\w.}=\Gamma_{\w.,z,>} \subset \N \times \N^n$
by using $\nu = \nu_{z,>} : \calo_{X,p} \setminus 0 \rarw \N^n$ as in Introduction.

In \cite[Lemma 2.12]{LM},
they assume that $\w.$ satisfies  ``Condition (C)'',
which seems to be a slightly stronger condition than being birational(= Condition (B) in \cite{LM}),
to show that $\Gamma_{\w.,z,>}$ generates $\Z \times \Z^n$ as a group for {\it any} $z$.
But we can show that it is enough to assume $\w.$ is birational.

\begin{lem}\label{cond(B) is enough}
Let $\w.$ be a birational graded linear series associated to a line bundle $L$ on a variety $X$.
Then $\Gamma_{\w.,z,>}$ is birational for any local coordinate system $z$ at any smooth point $p \in X$
and any monomial order $>$ on $\N^n$.
\end{lem}

\begin{proof}
The condition i) in Definition \ref{def of bir for semigroup} is clearly satisfied.
Thus it is enough to show that $\Gamma_{\w.}$ generates $\Z \times \Z^n$ as a group.
Fix $k \gg 0$.
Then the function field $K(X)$ is generated by $\{ \, f / g \in K(X) \, | \, f,g \in W_k, g \not = 0  \}$ over $\C$
because $\w.$ is birational.
Hence for any $F \in K(X) \setminus \{0\}$,
we can write $F= G/H$,
where $G,H$ are written as some polynomials over $\C$
of some elements in $\{ \, f / g \in K(X) \, | \, f,g \in W_k, g \not = 0  \}$.
Therefore we can write $F=G'/H'$ for some $G',H' \in W_{kl}$ for some $l \in \zp$.
Thus $\nu(F)= \nu(G')-\nu(H') \in \nu(W_{kl})-\nu(W_{kl}) \subset \Z^n. $
Since the valuation $\nu : K(X) \setminus \{0\} \rarw \Z^n$ is surjective (note that $\nu$ is naturally extended to $K(X) \setminus \{0\}$),
the group $\Z^n$ is generated by  $\{ \nu(W_{kl})-\nu(W_{kl}) \}_{l \in \N}$.
Thus the subgroup $\{0\} \times \Z^n$ in $\Z \times \Z^n$ is generated by
$\{0\} \times \{ \nu(W_{kl})-\nu(W_{kl}) \}_l \subset \Gamma_{\w.} - \Gamma_{\w.}$.

On the other hand,
$s_k \in W_k \setminus \{0\}$ and $ s_{k+1} \in W_{k+1} \setminus \{0\}$ induce the element
$(1,\nu(s_{k+1})-\nu(s_k)) \in \Gamma_{\w.} - \Gamma_{\w.} \subset \Z \times \Z^n$.
Since the group $\Z \times \Z^n$ is generated by $\{0\} \times \Z^n$ and $(1,\nu(s_{k+1})-\nu(s_k))$,
the semigroup $\Gamma_{\w.}$ generates $\Z \times \Z^n$ as a group.
\end{proof}

Let $\w.$ be a birational graded linear series associated to a line bundle $L$ on $X$.
We define the Okounkov body $\Delta(\w.)=\Delta_{z,>}(\w.)$ of $\w.$ with respect to $z$ and $>$ as in Introduction.
That is, $\Delta_{z,>}(\w.) = \Delta (\Gamma_{\w.,z,>}  )$.
Thus $\Delta(\w.)$ is an $n$-dimensional closed convex set in $\R^n$
because $\Gamma_{\w.} $ is birational by Lemma \ref{cond(B) is enough}.

\vspace{2mm}

The following lemma is similar to \cite[Lemma 5.2]{An}.

\begin{lem}\label{existence of alpha}
For a monomial order $>$ on $\N^n$
and a finite set $S$ in $\N^n$,
there exists $\alpha \in \zp^n$ satisfying the following:
For $u \in S$ and $v \in \N^n$ such that $v > u$,
it holds that $\alpha \cdot v > \alpha \cdot u$,
where $\alpha \cdot u, \, \alpha \cdot v$ are the usual inner products.
\end{lem}

\begin{proof}
For each $u \in S$,
set $S_u = \{ v \in \N^n \, | \, v > u\}$.
Let $I_u$ be the ideal in the polynomial ring $\C[\N^n]=\C[x_1,\ldots,x_n]$
generated by $\{ \, x^v \, | \, v \in S_u\}$.
By Hilbert's basis theorem,
$I_u$ is generated by $x^{v_{u1}},\ldots,x^{v_{uk_u}}$ for some $k_u \in \N$
and $v_{u1},\ldots,v_{uk_u} \in S_u $.
Therefore any $v \in S_u$ is contained in $v_{uj} + \N^n$ for some $j$.

We use the following result by Robbiano.

\begin{thm}[{\cite[Theorem 2.5]{Ro}}]\label{monomial order and weight}
For a monomial order $>$ on $\N^n$,
there exist an integer $s \in \{1, \ldots, n \}$
and $u_1 , . . ., u_s \in \R^n$
which satisfy the following:
For $u,v \in \N^n$,
$v > u $ if and only if $ \pi(v) >_{\lex}  \pi(u)$,
where
\[
\pi :\N^n \rarw \R^s ; \ u \mapsto (u_1 \cdot u,\ldots, u_s \cdot u) 
\]
and $>_{\lex}$ is the lexicographic order on $\R^s$.
\end{thm}

Let $e_1,\ldots,e_n$ be the standard basis of $\Z^n$
and consider the above $u_1,\ldots,u_s$ and $\pi$.
For any $\gamma >_{\lex} \delta$ in $\R^s$,
the following holds from the definition of the lexicographic order:
For $ \beta=(\beta_1,\ldots,\beta_s) \in \rp^s$ such that $\beta_1\gg \cdots \gg\beta_s$,
we have $\beta \cdot \gamma > \beta \cdot \delta$.

Hence we can take $\beta \in \rp^s$ such that
\begin{itemize}
\item $\beta \cdot \pi(e_i) > 0 $ for $1 \leq i \leq n$,
\item $\beta \cdot \pi(v_{uj}) > \beta \cdot \pi(u) $ for $u \in S$ and $ 1 \leq j \leq k_u$,
\end{itemize}
since $\pi(e_i) >_{\lex} \pi(0)=0 $ and $\pi(v_{uj}) >_{\lex} \pi(u) $
(we can take common $\beta$ because $S$ is a finite set).
Since $\beta \cdot \pi(u') = (\beta_1 u_1 + \cdots + \beta_s u_s) \cdot u'$ for $u' \in \N^n$,
we have
\begin{align}
\tag{$*$}
\alpha' \cdot e_i > 0, \quad \alpha' \cdot v_{uj} > \alpha' \cdot u
\end{align}
for $1 \leq i \leq n,u \in S$ and $ 1 \leq j \leq k_u$
if we take $\alpha' \in \Q^n$ sufficiently close to $ \beta_1 u_1 + \cdots + \beta_s u_s $.
Set $\alpha:=N \alpha' \in \Z^n$ for a sufficiently divisible positive integer $N$.
By $(*)$, it follows that $\alpha \in \zp^n$ and $\alpha \cdot v_{uj} > \alpha \cdot u$ for $u \in S$ and $j$.

We show this $\alpha$ satisfies the condition in the statement of this lemma.
Fix $u \in S$ and $v \in \N^n$ such that $v > u$,
i.e.,
$v \in S_u$.
Then $v \in v_{uj} + \N^n$ for some $1 \leq j \leq k_u$.
Thus we have $\alpha \cdot v = \alpha \cdot v_{uj} + \alpha \cdot (v-v_{uj}) \geq  \alpha \cdot v_{uj} > \alpha \cdot u $.
\end{proof}

\subsection{Proof of Theorem \ref{main}}

Since Seshadri constants are lower semicontinuous (cf.\  \cite[Example 5.1.11]{La}),
degenerations are useful to bound Seshadri constants from below.
For example,
Biran \cite{Bi} degenerates varieties to reducible schemes,
and gives lower bounds of (multi-point) Seshadri constants on $\P^2$.
In \cite{It},
the author uses toric degenerations to obtain lower bounds of Seshadri constants on some non-toric varieties.

We also use degenerations to prove Theorem \ref{main}.
Although we would like to degenerate $\bigoplus_k W_k$ to $\C[\Gamma_{\w.}]=\bigoplus_k V_{\nu(W_k)}$,
the semigroup $\Gamma_{\w.}$ is not finitely generated in general,
even if $\bigoplus_k W_k $ is finitely generated.
Instead,
we degenerate $W_k$ to $V_{\nu(W_k) }$ for each $k$ separately.
Then we can use the lower semicontinuity of jet separation, as in \cite{CM}.
This is the reason why we define Seshadri constants in terms of jet separation.

\vspace{2mm}
First,
we show that the Seshadri constants of $W(\Gamma)_{\bullet} $ and $W_{\Delta(\Gamma),\bullet}$ coincide
for a birational semigroup $\Gamma$.

\begin{lem}\label{s for semigroup}
For a birational semigroup $\Gamma \subset \N \times \N^n$ and $\mb \in \rp^r$,
we have $\ep(\Delta(\Gamma);\mb)=\ep(W(\Gamma)_{\bullet};\mb)$.
\end{lem}

\begin{proof}
By definition,
it holds that
\[
\ep(\Delta(\Gamma);\mb)
=\lim \dfrac{ s(V_{k \Delta(\Gamma)};\mb)}k, \quad
\ep(W(\Gamma)_{\bullet};\mb)  
= \lim \dfrac{ s(V_{\Gamma_k};\mb)}k.
\]
Since $\Gamma_k$ is contained in $k \Delta(\Gamma)$ for any $k$,
$\ep(\Delta(\Gamma);\mb) \geq \ep(W(\Gamma)_{\bullet};\mb)$ is clear.
Hence it is enough to show the opposite inequality.

When $\Gamma$ is finitely generated,
there exists $\tilde{u} =(l,u) \in \Gamma \subset \N \times \N^n$ such that
$(\Sigma(\Gamma) +\tilde{u} ) \cap (\N \times \N^n) \subset \Gamma$
by \cite[\S 3, Proposition 3]{Kh} (see also \cite[Subsection 2.1]{LM}).
This induces $(k  \Delta(\Gamma) \cap \N^n) +u \subset \Gamma_{k+l}$,
from which we have $ s(V_{k  \Delta(\Gamma)};\mb)= s( V_{(k  \Delta(\Gamma) \cap \N^n) +u};\mb) \leq  s(V_{\Gamma_{k+l}};\mb)$.
Thus we obtain
\begin{align*}
\ep(W(\Gamma)_{\bullet};\mb) = \lim_k \dfrac{ s(V_{\Gamma_k};\mb)}k
&= \lim_k \dfrac{ s(V_{\Gamma_{k+l}};\mb)}k  \\
&\geq \lim_k \dfrac{ s(V_{k  \Delta(\Gamma)};\mb)}k =\ep(\Delta(\Gamma);\mb).
\end{align*}

In the general case,
we take an increasing sequence $\Gamma^1 \subset \Gamma^2 \subset \cdots \subset \Gamma$
such that each $\Gamma^i$ is a finitely generated birational subsemigroup of $\Gamma$
and $\bigcup_{i\geq 1} \Gamma^i = \Gamma$.
Then it is easy to show
$\bigcup_i \Delta(\Gamma^i)^{\circ} = \Delta(\Gamma)^{\circ }$.
By Lemmas \ref{interior} and \ref{s=lim s_i}, we have
\[
\ep(\Delta(\Gamma);\mb) = \ep(\Delta(\Gamma)^{\circ };\mb)
= \lim \ep(\Delta(\Gamma^i)^{\circ};\mb) = \lim \ep(\Delta(\Gamma^i);\mb).
\]
Since each $\Gamma^i$ is finitely generated,
we can apply the first part of the proof of this lemma.
Hence we have
\[
\ep(\Delta(\Gamma^i);\mb) = \ep(W(\Gamma^i)_{\bullet};\mb) \leq \ep(W(\Gamma)_{\bullet};\mb).
\]
Thus we obtain $\ep(\Delta(\Gamma);\mb) = \lim \ep(\Delta(\Gamma^i);\mb) \leq  \ep(W(\Gamma)_{\bullet};\mb)$.
\end{proof}

Broadly speaking, the geometrical meaning of Lemma \ref{s for semigroup}
is that Seshadri constants of ample line bundles (on non-normal toric varieties) at very general points do not change by normalizations.
In fact,
when $\Gamma$ is finitely generated,
$(\Proj \bigoplus_k W_{\Delta(\Gamma),k},\calo(1))=(\Proj \C [\Sigma(\Gamma) \cap (\N \times \N^n) ],\calo(1))$
is nothing but the normalization of the toric variety $(\Proj \bigoplus_k W(\Gamma)_{k},\calo(1))=(\Proj \C [\Gamma],\calo(1))$.

\vspace{3mm}
Now we show the key proposition by using degeneration of global sections as \cite{CM}.
Roughly speaking,
this proposition states that $W \subset H^0(X,L)$ generically separates no less jets than $V_{\nu(W)}$.

\begin{prop}\label{key prop}
Let $L$ be a line bundle on an $n$-dimensional variety $X$.
Let $\nu=\nu_{z,>}$ be the valuation map defined by a local coordinate system
$z=(z_1,\ldots,z_n)$ at a smooth point $p \in X$
and a monomial order $>$ on $\N^n$.
Then $s(W;\mb) \geq  s(V_{\nu(W)};\mb)$ holds
for any subspace $W$ of $H^0(X,L)$ and any $\mb \in \rp^r$.
\end{prop}

\begin{proof}
By considering an increasing sequence of finite dimensional subspaces in $W$,
we may assume $\dim W < + \infty$.

Let $\pi : U \rightarrow \C^n$ be the \'etale morphism defined by $z_1,\ldots,z_n$
in an open neighborhood $U \subset X$ of $p$.
By the morphism $\pi$,
we can identify $\calo_{X,p}^{an}$ with $\calo^{an}_{\C^n,0}=\C \{x_1,\ldots,x_n\}$,
where $x_1,\ldots,x_n$ are the coordinates on $\C^n$
such that $\pi^*x_i=z_i$.
Then we can regard $W$ as a subspace of $\C \{x_1,\ldots,x_n\}$
by $W \hookrightarrow L_p \cong \calo_{X,p} \hookrightarrow \C \{x_1,\ldots,x_n\}$.
Note that $\nu$ is extended to
$\calo_{X,p}^{an} \setminus \{0\}=\C \{x_1,\ldots,x_n\} \setminus \{0\} \rightarrow \N^n$ naturally.

Choose and fix $f_u \in \nu^{-1}(u) \cap W$ for each $u \in \nu(W)$.
Then it holds that $V=\bigoplus_{u \in \nu(W)}\C f_u$ because \#$ \nu(W)=\dim W$ (cf.\  \cite{LM} or \cite{BC}).
Since $\nu(W)$ is a finite set,
there exists $\alpha \in \zp^n$ satisfying the following by Lemma \ref{existence of alpha};
if $v > u$ for $u \in \nu(W)$ and $v \in \N^n$, 
it holds that $\alpha \cdot v > \alpha \cdot u$.

The vector $\alpha$ induces the action $\circ$ of $\*c$ on $\C \{x_1,\ldots,x_n\}$
by $t \circ x^u := t^{\alpha \cdot u} x^u$ for $t \in \*c$ and $u \in \N^n$.
For $f_u=\sum_v c_{uv} z^v=\sum_v c_{uv} x^v$ (note we identify $z_i$ and $x_i$),
the regular function
\[
\dfrac{t \circ f_u}{t^{\alpha \cdot u}}=t^{-\alpha \cdot u} \sum_v c_{uv} t^{\alpha \cdot v} x^v
= \sum_v c_{uv} t^{\alpha \cdot v - \alpha \cdot u} x^v
\]
on a neighborhood of $\*c \times \{0\}$ (in $\*c \times \C^n$) 
is naturally extended to a regular function on a neighborhood $\mathcal{U}$ of $\C \times \{0\}$.
Note that $\alpha \cdot v -\alpha \cdot u \geq 0$ if $c_{uv} \not=0$.
We denote the regular function by $F_u$.
Set $\mathcal{W} = \bigoplus_{u \in \nu(W)} \C F_u$.

We prove this proposition only for $r=1$.
When $r >1$,
the proof is similar.
Thus we leave the details to the reader.

Choose a very general section $\sigma$ of the projection $\mathcal{U} \rightarrow \C$ onto the first factor.
Let $\mathcal{I}$ be the ideal sheaf corresponding to
$\sigma(\C)$ on $\mathcal{U} \subset \C \times \C^n$.
For $s \geq 0$,
we consider the map
\[
\phi: \mathcal{W} \otimes_{\C} \C\{t\} \rightarrow \calo^{an}_{\mathcal{U}}
\rightarrow \calo^{an}_{\mathcal{U}}/\mathcal{I}^{s+1}
\]
of flat sheaves over $\C$.
For $ t \in \C$,
we write $W_t:=\mathcal{W} \otimes \C\{t\} |_{\{t\} \times \C^n}$
and
\[
\phi_t:=\phi |_{\{t\} \times \C^n}: W_t \rarw \calo^{an}_{\mathcal{U}}/\mathcal{I}^{s+1} |_{\{t\} \times \C^n}= \calo_{\C^n} / \mathfrak{m}_{\sigma(t)}^{s+1}.
\]
By the flatness,
$\phi_t$ is surjective for very general $t$
if so is $\phi_0$.
Thus
if $W_0$ separates $s$-jets at $\sigma(0)$,
then $W_t$ also separates $s$-jets at $\sigma(t)$ for very general $t$.
Since $\sigma$ is a very general section,
we have
\[
s(W_t;\mb) \geq s(W_0;\mb)
\]
for $t$ in a neighborhood of $0$.
Since there is a natural identification
of $W_t$ and $W$ for $t \in \*c$ by the action $\circ$,
we have $s(W;\mb) = s(W_t;\mb)$.
On the other hand,
$W_0=V_{\nu(W)} \subset \C[x_1,\ldots,x_n]$
since $F_u=c_{uu}x^u + t \cdot (\text{higher term})$ for some $c_{uu} \not= 0$.
From these inequalities,
we have $s(W;\mb) \geq  s(V_{\nu(W)};\mb)$.
\end{proof}

Now we can show the main theorem easily.

\begin{thm}[=Theorem \ref{main}]\label{main thm}
Let $\w.$ be a birational graded linear series associated to a line bundle $L$
on an $n$-dimensional variety $X$.
Fix a local coordinate system $z=(z_1,\ldots,z_n)$ on $X$ at a smooth point
and a monomial order $>$ on $\N^n$.
Then $\ep(\w.;\mb) \geq \ep(\Delta_{z,>}(\w.);\mb)$ holds for any $r \in \zp$ and $\mb \in \rp^r$.
\end{thm}

\begin{proof}

Let $\Gamma := \Gamma_{\w.,z,>} \subset \N \times \N^n$ be the semigroup defined by $\w.,z$, and $>$.
Then we have $\Gamma_k=\nu(W_k) \subset \N^n$ and $\Delta(\Gamma) = \Delta_{z,>}(\w.)$ by definition.
By Proposition \ref{key prop},
it holds that
\[
\ep(\w.;\mb)= \lim \dfrac{s(W_k;\mb)}k \geq
\lim \dfrac{ s(V_{\nu(W_k)};\mb)}k=  \lim \dfrac{ s(V_{\Gamma_k};\mb)}k.
\]
Since $\Gamma$ is birational by Lemma \ref{cond(B) is enough},
we have
\[
\lim \dfrac{ s(V_{\Gamma_k};\mb)}k= \ep(W(\Gamma)_{\bullet};\mb)=\ep(\Delta(\Gamma);\mb)
=\ep(\Delta_{z,>}(\w.);\mb)
\]
from Lemma \ref{s for semigroup}.
Thus $\ep(\w.;\mb) \geq \ep(\Delta_{z,>}(\w.);\mb)$ holds.
\end{proof}

\begin{rem}
The inequality in Theorem \ref{main thm} is not equality in general.
Let $L$ be an ample line bundle
on a projective variety $X$.
Let $z$ be a local coordinate system at $p$
and $E$ the prime divisor on $X$ defined by the first coordinate $z_1$ of $z$ around $p$.
If $>$ is the lexicographic order,
$\Delta_{z,>}(L)$ is contained in $[0,a] \times \R^{\dim X-1}$,
where $a= \sup \{ \,t > 0 \, | \, L -t E \text{ is effective}\}$ (cf.\ \cite{LM}).
Hence $ \ep(\Delta_{z,>}(L);1 ) \leq a$ holds (cf.\ \cite[Theorem 3.6]{It}).
If we choose $z$ so that $E \in |mL|$ for $m \gg 0$,
we have $ \ep(\Delta_{z,>}(L);1 ) \leq a = 1/ m < \ep(L;1)$.
\end{rem}

\begin{rem}
See \cite[Remark 5.5]{LM} for another relation between Okounkov bodies and Seshadri constants,
though the relation is not written explicitly there.
The relation also holds for a birational graded linear series.
\end{rem}

\end{document}